\documentclass[11pt]{article}
\usepackage{amsmath,amssymb,latexsym,float,epsfig,subfigure}
\usepackage{latexsym}
\usepackage{amsthm}
\usepackage{epsfig}
\usepackage{algpseudocode}
\usepackage{subfigure}
\usepackage{scalefnt}
\usepackage{booktabs}
\usepackage[english]{babel}
\usepackage{graphicx}
\usepackage{titlesec}
\usepackage{xcolor,colortbl}
\usepackage{caption}
\usepackage[normalem]{ulem}
\usepackage{epsfig}
\usepackage{enumerate}
\usepackage{caption}
\usepackage{subfigure}
\usepackage{booktabs}
\usepackage[english]{babel}
\usepackage{graphicx}
\usepackage{multirow}
\usepackage{rotating}
\usepackage{fancyhdr}
\usepackage{setspace}
\usepackage[section]{placeins}
\usepackage{breqn}
\usepackage{pseudocode}
\usepackage{algorithmicx}
\usepackage{algorithm}
\usepackage{url} 
\usepackage{appendix}
\usepackage{xcolor}

\usepackage{mathtools}

\newtheorem{lemma}{Lemma}

\newtheorem{proposition}{Proposition}
\newtheorem{definition}{Definition}

\newcommand{\PRTRP}[1]{PRTRP}

\def \Re{\mathbb{R}}

\topmargin by -\headsep \textheight 8.9in \oddsidemargin 0pt
\evensidemargin \oddsidemargin \marginparwidth 0.5in \textwidth
6.5in
\begin{document}
\title{Repair Crew Routing for Power Distribution Network Restoration}

\author{Bahar \c{C}avdar \thanks{bcavdar@tamu.edu} \and Qie He \thanks{qiehe01@gmail.com} \and Feng Qiu \thanks{fqiu@anl.gov}}
\date{%
    $^*$\small{Department of Engineering Technology and Industrial Distribution, Texas A\&M University, College Station, TX 77843}\\%
    $^\dagger$Department of Industrial and Systems Engineering, University of Minnesota, Minneapolis, MN 55455\\[2ex]%
    $^\ddagger$Argonne National Laboratory, Lemont, IL 60439
}

%\date{}
\maketitle
\begin{abstract}
{
As extreme weather events become more frequent and disruptive, service restoration is increasingly important for many infrastructures, e.g., power grids and  communication networks. In many studies on service restoration, the logistics issue of traveling over the road network, however, is often overlooked due to the complexity of considering multiple networks simultaneously, resulting in prolonged disruption time. In this work, we address such a problem arising in power systems, where technical crew and utility trucks travel to a number of sites to repair damaged equipment, with the goal of minimizing the total service disruption time within the service region. We call this problem the Power Restoration Traveling Repairman Problem (\PRTRP{}). What makes it significantly more challenging than a typical routing problem is that the service disruption time in a location depends on the interaction of the routing sequence with both networks, i.e., the road network and the power grid. To solve the problem, we develop an exact method based on bi-directional dynamic programming. We then improve the method by reducing the search space with solution upper and lower bounds, and threshold rules derived from the precedence relations in the power grid. We also propose efficient heuristic variants of the method. We present computational results and compare our method with benchmark heuristics.
}
\end{abstract}

%\doublespacing

\section{Introduction}
In recent years, extreme weather events and natural hazards, such as hurricanes, wildfires, and storms, have been causing a growing number of significant disruptions to various critical infrastructures. With severe physical damages, service restoration could take days, weeks, or even months, causing an enormous economic loss. Recent examples include weeks-long power outage in Florida in 2017 by Hurricane Irma \cite{WP} and months-long power outage in Puerto Rico in 2017 caused by Hurricane Maria \cite{NPR}. Poor logistics decisions could slow down restoration and prolong disruption time significantly. For instance, in electricity service restoration, technical crew and utility trucks need to travel to a number of sites to repair damaged equipment and install new equipment or perform safety inspection before the switch can be closed and service restored. Therefore, the route in the road network and the restoration activity in the power network are interdependent. Ideally, the restoration activities should be considered together with the logistics decision. However, since the service restoration itself is extremely complicated \cite{QIURestoration}, the logistics issue is often ignored or considered in an isolated setting, resulting in sub-optimal solutions with longer service disruption times.

In this work, we address service restoration with logistics considerations. More specifically, we focus on power system restoration with repair crew routing decisions. Our problem is to determine the route of a single repair crew to visit and repair the disrupted nodes on a power distribution network to ensure that power service is restored in a timely manner. While this problem is very similar to the Traveling Repairman Problem (TRP), it requires a new solution strategy due to a complicated objective function, which arises from the precedence relations imposed by the power network. In this paper, we first formulate this problem as a mixed-integer program (MIP). Solving this problem using off-the-shelf optimization solvers takes a long time for instances of small sizes. Therefore, we focus on developing a more efficient exact solution method based on bi-directional dynamic programming (BiDP). Our method is further strengthened by reducing the search space with structural results derived from the power and road networks. We also propose heuristic implementations of the BiDP approach to find solutions faster. With the heuristic approach, we find optimal solutions in almost all test instances with a significant reduction in computation time.

The rest of the paper is organized as follows. In Section \ref{Sec:Lit}, we review the related literature. In Section \ref{Sec:Def}, we formally introduce the problem and present the MIP formulation. In Section \ref{Sec:Method}, we describe the BiDP approach along with structural results to reduce the search space and primal heuristics. In Section \ref{Sec:Comp}, we present the computational results of both the exact and heuristic approaches. In Section \ref{Sec:Conc}, we summarize our contribution and present some future research directions.

\section{Literature Review}
\label{Sec:Lit}
Our work is related to network resilience, restoration of power networks, and routing problems.

Increasing frequency and impact of power service disruptions have attracted more attention to  resilience in power systems in recent years. \cite{jufri2019state} provides a review on power system resilience. They cover the terminology related to resilience, present a resilience framework, and discuss alternative courses of actions to improve resilience. In general, efforts to increase resilience can be divided into two phases: \emph{preparedness} and \emph{response}. Preparedness activities generally focus on increasing the strength of the underlying system and positioning the resources in anticipation of a disruption. Some examples of preparedness activities are selective undergrounding, physical upgrading, substation relocation and line rerouting, emergency generators and mobile substations, spare parts, and crew management \cite{jufri2019state}. \cite{binato2001new} and \cite{bienstock2007using} are two example studies focusing on preparedness through network design. The former considers a transmission network design problem, and the latter focuses on arc capacity increases. \cite{arab2015proactive} studies the prediction of the impact of a potential disaster and pre-disaster crew mobilization. Response actions, on the other hand, take place in the aftermath of a disruption. Most of the response strategies revolve around the management of repair crew to restore the disrupted parts on the network to ensure the power service is restored as quickly as possible. Our study mainly focuses on response strategies.

The literature on disaster recovery for power systems is relatively young, with most of the studies published after 2010 \cite{chen2018toward, tan2019scheduling}. These studies can be divided into two groups based on the part of the power network considered: distribution network restoration and transmission network restoration. Although the problems in both networks are similar due to the underlying precedence relations, the high voltage on the transmission network requires a precise model of the power flow problem. %We first review the literature on power distribution network restoration. %Given that upgrading distribution networks which cover urban regions is more expensive than transmission networks, efficient recovery strategies will remain critical. 
\cite{arif2017power} is one of the early works on addressing the routing issue of repair crew on distribution networks. They propose a two-stage model, where the first stage allocates damaged locations to clusters and the second stage determines the route in each cluster. The number of disruptions in their test instances goes up to 17. \cite{thiebaux2013planning} follows a different approach. In case of a disruption, they first solve the reconfiguration problem to ensure all nodes are supplied, and then solve the restoration problem to determine the order of repairs without considering the routing decisions. \cite{tan2019scheduling} emphasizes the difference between disruption restoration and blackouts restoration. Blackouts generally occur due to an instability in the power system, and the service restoration involves determining the sequence of switching actions to re-energize the demand nodes. On the other hand, disruptions caused by external factors such as natural disasters need to be attended by a repair crew. The authors study the distribution network repair and restoration problem, where travel time is ignored and a repair schedule is found with heuristics based on linear programming relaxation. \cite{chen2018toward} emphasizes the routing aspect of the distribution network restoration. They consider two groups of crews with different capabilities and propose an integer programming model to determine routes for the repair crew. % --> this paper does not consider the flow of the current.
To our knowledge, \cite{arif2018optimizing} is the only study considering the uncertainty in distribution network restoration. They assume that disruption locations are known, but travel times and demand are uncertain. In their two-stage model, they first determine the routes for the repair crew and then the reconfiguration scheme to re-energize the demand nodes.

In transmission network restoration, \cite{van2011vehicle} and \cite{simon2012randomized} %are the first studies
focus on the last-mile disaster recovery for power restoration. They consider a disrupted transmission network to determine the schedules and routes for the repair crew to restore the power as quickly as possible. They define this problem as Power Restoration Vehicle Routing Problem. The authors solve the problem in three steps. In the first step, they determine the minimum set of disrupted items to repair. In the second step, they determine the order of restoration without considering routing. In the third step, they determine the routes for the repair crew considering the partial ordering imposed by the first step. %They present the mathematical formulations and use a large neighborhood search (LSNS) heuristic.
\cite{coffrin2015transmission} studies the same problem with a focus on comparing the performances of two different power flow models, namely direct current and linear approximation of alternating current (LPAC). \cite{van2015transmission} considers the transmission network restoration problem using the LPAC model, with a similar formulation and solution methodology to that in \cite{coffrin2015transmission}. \cite{coffrin2012last} is the first study on the last-mile restoration problem considering the inter-dependency between the power and gas networks. They solve the problem in three steps and use a randomized adaptive decomposition method. They consider disruption scenarios with 10 to 120 damaged locations. 
\cite{arab2016electric} determines the optimal repair schedule, unit commitment, and system reconfiguration to restore disrupted power networks. They develop a MIP formulation and use Benders decomposition. Their test instances are based on IEEE-118 instance with seven damaged nodes. \cite{gao2020optimal} proposes a model to determine the schedule to repair the disrupted transmission network. They ignore the routing decision of the repair crew, and test their model on small instances with six disrupted nodes.

Our underlying problem is a routing problem. Considering the cumulative nature of our objective function, the Cumulative Capacitated Vehicle Routing Problem (CCVRP) and the TRP are the two most closely related problems. \cite{lysgaard2014branch} provides the first exact solution method for the CCVRP based on a branch-and-cut-and-price algorithm. \cite{ribeiro2012adaptive}, \cite{ngueveu2010effective}, and \cite{rivera2015multistart} present metaheuristics for the CCVRP based on large neighborhood search, a memetic algorithm, and multi-start local search, respectively. \cite{wu2004exact} proposes an exact solution method based on dynamic programming and branch and bound. They find solutions for instances with up to 23 nodes. For the TRP, \cite{luo2014branch} proposes an exact solution approach for the multiple TRP based on a branch-and-price-and-cut method. They provide solutions for instances up to 50 nodes. \cite{archer2008faster} develops a 7.18-approximation for the TRP. 

One study that is most closely related to ours is \cite{morshedlou2018work}, in which the joint routing and scheduling problem is considered for restoring an infrastructure network. They propose two MIP models, one for binary restoration and one for proportional restoration, and develop heuristics to solve the models. \cite{Ulusan2021-su} considers a disrupted road network and develops a solution methodology based on approximate dynamic programming. While this paper has similarities to ours in terms of solution methodology, we consider a power network in parallel to a road network. \cite{Goldbeck2020-si} also emphasizes the logistics aspects of the response strategies, but they consider logistics as a part of a higher level planning with three geographical areas. 

The current literature on crew routing for repairing infrastructure networks is still very limited in terms of understanding the impact of precedence relations. \cite{Suryawanshi2022-ze} addresses this gap in the literature in a recent review. In this study, we develop a new solution approach for the single crew routing problem to minimize the negative impact of infrastructure network disruptions. In the following section, we describe our problem in more detail and present the formulation.

\section{Problem Definition, Formulation, and Complexity}
\label{Sec:Def}

Before presenting the problem, we first provide background information on power distribution networks as an example of infrastructure networks with precedence relations. A typical power grid consists of power stations, substations, the transmission network, and the distribution network. %Transmission network carries the high-voltage power over long distances. 
The distribution network is considered to be the last mile of the power distribution. It delivers the low-voltage power to final customers. There are different configurations of power network topology such as radial, multi-ring, and mesh structure, offering different resilience, costs, and complexities to manage. The radial network configuration, where there are no close loops, is the most common configuration for distribution networks \cite{prakash2016review}. While it is the simplest and cheapest topology, if a node on the network is disrupted, all downstream nodes lose power service. Figure \ref{Fig:13Feeder} shows the IEEE 13 bus feeder instance as an example of a power distribution network. On this network, node 50 is the source. If there is a disruption on node 84 only, in addition to 84, nodes 52 and 911 lose power as well. %On the other hand, transmission networks generally have loops. Therefore, managing transmission networks is more complex.
\begin{figure}[H]
    \centering
    {{\includegraphics[width=7cm]{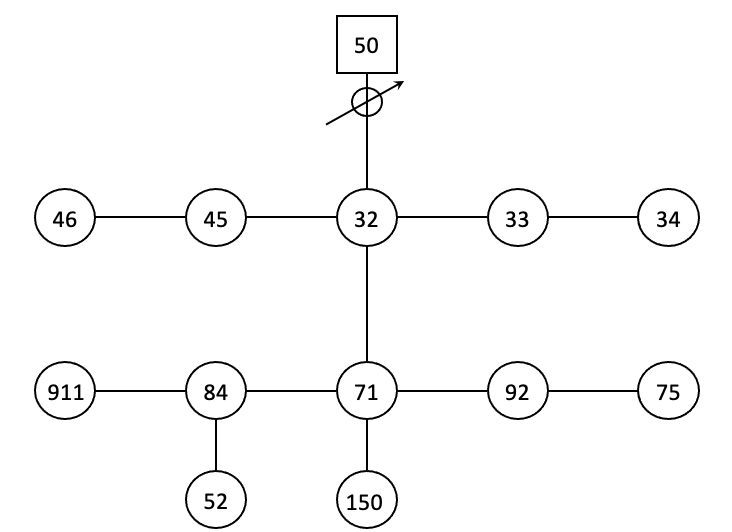} }}%
 \caption{The IEEE 13 bus feeder instance.}%
    \label{Fig:13Feeder}%
\end{figure}

We study how to efficiently restore a power distribution network where faults occur in some locations and result in outage for customers over the network. In this paper, we assume that faults happen only at the vertices on the power distribution network and their locations are known. The utility company sends out a repair crew to visit these fault locations sequentially and make repairs. The goal is to determine a sequence of fault locations to visit for the repair crew to minimize the total service disruption time for all customers.

\subsection{Problem definition}
There are two underlying networks in our problem: the road network on which the repair crew travels, and the power distribution network through which electricity is sent to customers. In our network model, we use a vertex to model both the intersection of two road segments in the road network and a fault location in the power network. We first introduce some notation for the two networks.

Let vertex 0 denote the depot where the repair crew is located before the repair starts. Let $V_c=\{1,2 \dots, n\}$ be the set of all fault locations (power substation or a customer location). All locations in $V_c$ receive power from the substation through distribution lines in the power distribution network, which is a directed acyclic graph. We consider the case in which there is only one substation in the network. For $i,j \in V_c$, an arc $(i,j)$ denotes that vertex $i$ is the immediate predecessor of vertex $j$ on the unique path from the power source to vertex $j$ in the power distribution network. The tree $T=(V_c, A_c)$ characterizes the precedence relations among all fault locations in the power network where the arc set $A_c$ is the collection of precedence relations on the power distribution network.
%Then the power distribution network can be represented as a directed tree $T=(V_c, A_c)$ with the root vertex being the substation and each arc $(i,j) \in A_c$ representing power lines connecting an upstream vertex $i$ to a downstream vertex $j$ in the distribution network. 
On the other hand, we model the road network as a complete directed graph $G=(V, A)$ where $V=V_c \cup \{0\}$ and each arc $(i,j)\in A$ represents a shortest path from location $i\in V$ to location $j\in V$. The travel time over arc $(i,j)$ is $d_{ij}$ for each $(i,j) \in A$.

%Let vertex 0 denote the depot where the repair crew is located before the repair starts. Let $V_c=\{1,2 \dots, n\}$ be the set of all fault locations (power substation or a customer location). All locations in $V_c$ receive power from substations through distribution lines in the power distribution network, which is a directed acyclic graph. We consider the case in which there is only one substation in the network. Then the power distribution network can be represented as a directed tree $T=(V_c, A_c)$ with the root vertex being the substation and each arc $(i,j) \in A_c$ representing power lines connecting an upstream vertex $i$ to a downstream vertex $j$ in the distribution network. On the other hand, we model the road network as a complete directed graph $G=(V, A)$ where $V=V_c \cup \{0\}$ and each arc $(i,j)\in A$ represents a shortest path from location $i\in V$ to location $j\in V$. The travel time over arc $(i,j)$ is $d_{ij}$ for each $(i,j) \in A$.

The main complexity of our problem compared to a typical routing problem comes from how the service disruption time at a fault location is computed. Let vertex $1\in V_c$ denote the substation (the unique power source in the power network). For each vertex $i\in V_c$, there exists a unique path $P_{1i}$ in $T$ from the substation to vertex $i$. The power service at vertex $i$ can only be recovered when all faults on the path $P_{1i}$ are removed. Therefore, the service disruption time at vertex $i$ is the duration between the time the repair crew leaves the depot (vertex 0) and the time all faults on path $P_{1i}$ (including the fault at vertex $i$) are removed. Let $r_i$ denote the time it takes until the recovery of power service at vertex $i\in V_c$. Note that $r_i$ depends on both the topology of the power network $T$ as well as the sequence in which fault locations are visited in the road network $G$. Our goal is to determine a Hamiltonian cycle over vertices $V$ starting at vertex 0 to minimize the total service disruption time $\sum_{i\in V_c}r_i$.

\subsection{Formulation}
Before introducing the formulation of \PRTRP{}, we assume that the repair duration $p_i$ at each vertex $i\in V_c$ is 0. To see why this assumption is without loss of generality, we can transform the original problem into an equivalent \PRTRP{} with zero repair duration at each vertex, by changing two inputs: set the travel time $d'_{ji}$ in the new problem to be $d_{ji}+p_i$ for each incoming arc $(j,i)\in A$ in the road network and $p'_i=0$. This transformation does not change the objective value of any repair sequence and therefore creates an equivalent \PRTRP{} with no repair duration. We also introduce some notation to assist the calculation of the service disruption time at a vertex. Given two vertices $i$ and $j$ in $V_c$, let $i \succ j$ denote that vertex $i$ is a predecessor of vertex $j$ in the power network $T$, i.e., vertex $j$ is reachable from vertex $i$ through a directed path on $T$. If vertex $i$ does not receive power service, neither does vertex $j$. Let $i \succeq j$ denote that either $i \succ j$ or $i = j$. Given integers $i, j$ with $i \le j$, we use $[i:j]$ to denote the set of integers $\{i, i+1, \ldots, j\}$.

We now introduce the formulation of the problem. Let the binary decision variables $x_{ij}$ denote whether the repair crew travels from vertex $i$ to vertex $j$ for each arc $(i, j) \in A$, the continuous decision variables $t_i$ denote the arrival time at vertex $i$ for each $i \in V$, and the continuous decision variables $r_i$ denote the service disruption time of vertex $i$ for each $i \in V_c$. Then the \PRTRP{} can be formulated as a MIP as follows.
\begin{align} 
\min \; & \sum_{i \in V_c} r_i \label{eq_mipformulation_obj}\\
\text{s.t.} \; & \sum_{j \in V \setminus \{i\}}x_{ij} = \sum_{j \in V \setminus \{i\}} x_{ji} = 1,  & \forall i \in V, \label{eq_mipformulation_degrees_constraints}\\
& t_j \ge t_i + d_{ij} - M(1-x_{ij}),  &\forall j \in V_c, i\in V, \label{eq_mipformulation_arrival_time_constraints}\\
& r_j \ge t_i ,  &\forall j\in V_c, \forall i \succeq j, \label{eq_mipformulation_disruption_time_constraints} \\
& x_{ij} \in \{0,1\},  &\forall (i,j) \in A,\\
%& r_j \ge 0, \; \forall j \in V_c,\\
& t_j \ge 0,  &\forall j \in V_c,\\
%&\text{subtour elimination constraints} \\ (Qie: we don't need subtour elimination constraints here.)
& t_0 = 0.
\end{align}

The objective function in \eqref{eq_mipformulation_obj} calculates the total service disruption time at all fault locations. Constraints \eqref{eq_mipformulation_degrees_constraints} enforce that each faulty vertex on the distribution network should be visited exactly once. Constraints \eqref{eq_mipformulation_arrival_time_constraints} ensure that vertex $j$ will be visited after vertex $i$ if the arc $(i,j)$ is used. Constraints \eqref{eq_mipformulation_disruption_time_constraints} are the unique constraints compared to other routing problems. They ensure that vertex $j$ will only receive the power service after all its predecessors and itself in the power grid are visited and repaired. We need to set the value of $M$ in \eqref{eq_mipformulation_arrival_time_constraints} large enough for the formulation to be valid, for example setting $M = \sum_{(k,l)\in A}d_{kl}$. It is possible to choose smaller values for $M$. However, it is well known that this type of big-M formulations provide weak linear programming relaxations and take a long time to solve for large-size instances. 

%\subsection{Complexity}
\textbf{Complexity: }It is not surprising that the described problem is NP-hard.
\begin{proposition}
The \PRTRP{} is NP-hard.
\end{proposition}
\begin{proof}
We will show a polynomial-time reduction from the NP-hard traveling repairman problem to the \PRTRP{}. Consider an instance of the traveling repairman problem over a graph $N = (V', A')$, where $V' = \{1, \ldots, n\}$ and the travel time over arc $(i,j) \in A'$ is $d_{ij}$. Construct an instance of \PRTRP{} in the following way. For the road network, set $V_c = V'$, put the depot 0 at the same location as vertex 1, and let the travel time from vertex $i$ to vertex $j$ be $d_{ij}$ for $(i,j) \in V_c$. For the power network, make vertex 1 the power source and each vertex $i \in V_c \setminus \{1\}$ the immediate successor of vertex 1; the power network is essentially a star with $n-1$ leaves. Finding an optimal solution to the traveling repairman problem over graph $N$ is equivalent to finding an optimal solution of the \PRTRP{} over the constructed road and power networks, and the transformation takes polynomial time.
\end{proof}

%In the next section, we present our solution methodology to find exact solutions more efficiently.

\section{Solution Methodology}
\label{Sec:Method}
To solve the \PRTRP{}, we develop an algorithm based on bi-directional dynamic programming (DP). Our algorithm performs backward and forward DP at the same time and uses the combined information to efficiently reduce the search space. It is also strengthened by pre-processing steps that are developed based on the precedence relations. Before presenting the details of our algorithm, we first introduce some definitions and new notation.

\begin{definition}
We call a path $P$ a return path if it starts at some vertex $i\in V_c$ and ends at vertex 0.
\end{definition}
\begin{definition}
We call a path $P$ an outgoing path if it starts at vertex 0 and ends at some vertex $i\in V_c$.
\end{definition}
Given a path $P$, let $V(P)$ denote the set of vertices on path $P$. Given two paths $P_1$ and $P_2$, let $P_1 \oplus P_2$ denote a path obtained by concatenating $P_2$ after $P_1$. For example, if $P_1 = (0, 3, 5, 4)$ and $P_2 = (1, 2)$, then $P_1 \oplus P_2 = (0, 3, 5, 4, 1, 2)$.
Assume that $P$ is a return path starting at vertex $i\in V_c$. Let $v^B(P)$ denote the service disruption time for all customers in $V_c$ between the time the vehicle arrives at vertex $i$ and the time the vehicle returns to vertex 0 following path $P$.
Given an outgoing path $P$ ending at vertex $i \in V_c$, let $u^F(P)$ denote the service disruption time for all customers between the time the vehicle leaves the depot and the time vehicle arrives at vertex $i$ following path $P$.
We introduce some key concepts and results that will be used in DP.

\begin{definition}
Two return paths $P_1$ and $P_2$ have the same configuration if $P_1$ and $P_2$ start at the same vertex and $V(P_1) = V(P_2)$. Similarly, two outgoing paths $P_1$ and $P_2$ have the same configuration if $P_1$ and $P_2$ end at the same vertex and $V(P_1) = V(P_2)$. 
\end{definition}

\begin{proposition}
If two return paths $P_1$ and $P_2$ have the same configuration and $v^B(P_1) < v^B(P_2)$, then no optimal Hamiltonian cycle of the \PRTRP{} ends with $P_2$.
\end{proposition}

\begin{proof}
Suppose that there exists an optimal Hamiltonian cycle $P$ that ends with $P_2$. Assume that $P=P_3\oplus P_2$. Then $P^'=P_3 \oplus P_1$ is a Hamiltonian cycle leading to a smaller total service disruption time than that of $P$, which contradicts the optimality of path $P$.
\end{proof}

\begin{proposition}
If two outgoing paths $P_1$ and $P_2$ have the same configuration and $u^F(P_1) < u^F(P_2)$, then no optimal Hamiltonian cycle of the \PRTRP{} starts with $P_2$.
\end{proposition}

\begin{proof}
Suppose that there exists an optimal Hamiltonian cycle $P$ that starts with $P_2$. Assume that $P=P_2 \oplus P_3$. Then $P^'=P_1 \oplus P_3$ is a Hamiltonian cycle leading to a smaller total service disruption time than that of $P$, which contradicts the optimality of $P$.
\end{proof}

In the BiDP algorithm, we construct outgoing paths and return paths simultaneously, and use additional bounding techniques to reduce the search space.

\subsection{Backward dynamic programming}
We denote the state vector of the backward DP by $(k, Y)$, where vertex $k$ is the current vertex, i.e., the current location of the repair crew, and $Y\subseteq V_c$ is the set of vertices that are not visited (repaired) yet. Note that $k \notin Y$. Let $v(k, Y)$ denote the minimum total service disruption time of all customers between the time the crew leaves vertex $k$ and the time the crew finishes removing all faults. Then we have the following Bellman equations.
\begin{align}
v(k, \emptyset) = 0, & \qquad \forall k \in V_c, \label{eq:backward:base} \\
v(k, Y) = \min_{j \in Y} \{W(V_c \setminus Y) d_{kj} + v(j, Y \setminus \{j\})\}  & \qquad Y \subseteq V_c, |Y| \ge 1, \label{eq:backward:inductive}
\end{align}
where $W: 2^{V} \mapsto \mathbb R_{> 0}$ and $W(S)$ computes the number of vertices whose power services are disrupted when the set of vertices in $S$ are repaired. Equation~\eqref{eq:backward:base} simply states that when the crew is at vertex $k$ and there is no vertices left to visit, then the minimum service disruption time of all customers between the time the crew leaves vertex $k$ and the time the crew finishes removing all faults is 0. Equation~\eqref{eq:backward:inductive} states that the minimum service disruption time of all customers between the time the crew leaves vertex $k$ and the time the crew finishes removing all faults equals to the smallest sum of two terms: the first term $W(V_c \setminus Y) d_{kj}$ denotes the additional service disruption time of all the vertices when the crew travels from vertex $k$ to vertex $j$, and the second term $v(j, Y \setminus \{j\})$ is the minimum service disruption time of all customers between the time the crew leaves vertex $j$ and the time the crew finishes removing all faults, for some vertex $j$.

Let $v^*$ be the minimum service disruption time among all Hamiltonian cycles. Then $v^*=v(0, V_c)$. A Hamiltonian cycle $(j_0 = 0, j_1, \dots, j_n, 0)$ (with distinct vertices $j_1, \ldots, j_n \in V_c$) is optimal if and only if
% Then,
% \begin{equation}
% v^*=v(V_c,0)=\min_{j\in V_c}\{W(\emptyset)d_{0j}+v (V_c\setminus\{j\},j)\}.    
% \end{equation}
%A Hamiltonian path $(0, j_1, j_2, \dots, 0)$ is optimal if and only if 
%$$W(\emptyset)d_{0j_1}+v(V_c \setminus \{j_1\}, j_1) = \min_{j \in V_c} \{W(\emptyset)d_{0j}+v(V_c \setminus \{j\}, j)\}$$
%and 
$$v(j_{p-1}, \{j_p, \dots, j_n\})=W( \{j_1, j_2, \dots, j_{p-1}\})d_{j_{p-1}, j_p}+ v(j_p, \{j_{p+1}, \dots, j_n\}) \quad \forall p \in [1:n].$$

%%%%%%%%%%%%%%%%%%%
\subsection{Forward dynamic programming}

We denote the state vector of the forward DP by $(X, k)$, where $X \subseteq V$ is the set of vertices that have been visited so far and vertex $k$ is the current location of the repair crew. Note that $k \in X$. Let $u(X,k)$ denote the minimum total service disruption time of all customers between the vehicle leaves the depot, visits all customers in $X$, and arrives at vertex $k$. We can write the Bellman equations for the forward DP as follows.
\begin{align}
    u(\{0\}, 0)=0, &  \label{eq:forward:base} \\
    u(X,i)= \min_{j \in X\setminus\{i\} } \{ u(X\setminus\{i\}, j) + W(X\setminus \{i\})d_{ji} \}, & \quad \{0\} \subsetneq X \subseteq V, i \in X \setminus\{0\}. \label{eq:forward:inductive}
\end{align}

Let $u^*$ be the minimum service disruption time among all Hamiltonian cycles. Then 
$u^*=\min_{i \in V_c}u(V, i)$. 
%Then,
%$$P^*= \min_{i \in V_c}u(V_c,i).$$
A Hamiltonian cycle $(j_0 = 0, j_1, \dots, j_n, 0)$ (with distinct vertices $j_1, \ldots, j_n \in V_c$) is optimal if and only if
\begin{align}
    u^*& = u(V, j_n),  \\
    u(\{j_0, j_1, \dots, j_{p+1}\}, j_{p+1})&= u(\{j_0, j_1, \dots, j_p\}, j_p) + W(\{j_0, j_1, \dots, j_p\})d_{j_p j_{p+1}}, p\in [0:n-1].
\end{align}

%%%%%%%%%%%%%%%%%%%%%%%%%%%
\subsection{Pre-processing and the maximum position of a vertex in optimal solutions}
%The precedence relation imposed by the distribution network increases the computational requirement of our problem. On the other hand, we can use this information to eliminate non-optimal solutions. 
In the pre-processing step, we determine the maximum position of a vertex in any optimal Hamiltonian cycle and use that to reduce the search space during BiDP.

We first introduce some notation that will be used in such bounds. Let $S_i$ be the set of successors of vertex $i$ on the power network including vertex $i$, i.e., $S_i=\{j \in V_c: i \succeq j\}$, and $|S_i|$ be the cardinality of set $S_i$. Let $s_p$ be the length of the $p^{th}$ shortest arc in the road network $G$. 
%Let $|S_i|= \sum_{j \in V_c}  \mathbbm{1}_{i \succeq j}$ denote the number of vertices in the distribution network that cannot receive service before the fault on vertex $i$ is recovered. 
\begin{proposition}
\label{Pro:lower}
Given a vertex $i \in V_c$, for any $k> n -|S_i|$, the total service disruption time of any Hamiltonian cycle that visits vertex $i$ in the $k^{th}$ position is lower bounded by $L_i^k$, where
\begin{equation}
\label{Eq:Threshold}
    L_i^k=\sum_{\substack{p : p \leq n-|S_i| \\ p\geq k+1}}(n-p+1)s_p + |S_i|\sum_{p=n-|S_i|+1}^{k}s_p.
\end{equation}
 \end{proposition}

The proof of Proposition~\ref{Pro:lower} is given in Appendix \ref{App:lower}. We can use the lower bound associated with vertex $i$ in Proposition~\ref{Pro:lower} to bound the maximum position of the vertex in an optimal solution.
 
%The following corollary provides a stronger lower bound but at an increased complexity.

% \begin{corollary}
% Let $h_p$ be the $p^{th}$ shortest arc length on the shortest Hamiltonian path on the set $V_c$. Consider a Hamiltonian cycle for the distribution network restoration where vertex $i\in V_c$ is the $k^{th}$ vertex visited after leaving the depot. If $k> n-|S_i|$, a lower bound $H_i^k$ for the total service disruption time following this Hamiltonian cycle can be computed as follows:
% \begin{equation}
% \label{Eq:Threshold2}
%     H_i^k=\sum_{\substack{p \leq n-|S_i| \\ p\geq k+1}}(n-p+1)h_p + |S_i|\sum_{p=n-|S_i|+1}^{k}h_p.
% \end{equation}
% \end{corollary}

\begin{proposition}
\label{Prop:Threshold}
Let $U$ be an upper bound for the optimal objective value of the \PRTRP{}. Given a vertex $i \in V_c$ and some position $k > n - |S_i|$, if the lower bound $L_i^k$ defined in~\eqref{Eq:Threshold} satisfies that $L_i^k > U$, then any optimal Hamiltonian cycle for the \PRTRP{} will not visit vertex $i$ at the $k^{th}$ position or later.
\end{proposition}

\begin{proof}
For the lower bound $L_i^k$ defined in~\eqref{Eq:Threshold}, if $L_i^k > U$, then it is not optimal to visit vertex $i$ in the $k^{th}$ position of a Hamiltonian cycle. 
To prove the proposition, it suffices to show that for a given vertex $i$, $L_i^k$ is monotonically non-decreasing in $k$. According to~\eqref{Eq:Threshold} , $L_i^{k+1}-L_i^{k}=|S_i|s_{k+1} - (n-k-1+1)s_{k+1} = (|S_i|+k-n) s_{k+1}$. Since $k > n-|S_i|$, $L_i^{k+1} > L_i^{k}$. By induction on $k$, $L_i^{k'}>L_i^{k}$ for each $k'>k$. Therefore, if it is suboptimal to visit vertex $i$ on the $k^{th}$ position, it is also suboptimal to visit it on later positions.
\end{proof}

The monotonicity of the lower bound $L_i^k$ in Equation (\ref{Eq:Threshold}) is due to the condition that $k > n-|S_i|$. Since $|S_i|> n-k$, some of the vertices visited before the $k^{th}$ vertex are successors of vertex $i$. These vertices cannot receive power before $i$ is recovered, so delaying the repair of $i$ delays their service recovery, thus increasing the associated lower bound. Proposition \ref{Prop:Threshold} implies that for vertex $i \in V_c$, if we find a position $k$  that satisfies $L_i^{k} > U$, then any solution where vertex $i$ is visited on $k$ or later could be removed from the search space of the dynamic programming. %In the following section, we present the bidirectional dynamic programming algorithm.
%%%%

\subsection{Heuristics bounds}
\subsubsection{Lower bounds of solutions extended from a given path}
We derive lower bounds for the objective value of Hamiltonian cycles that can be extended from an outgoing or return path. These lower bounds will be used to reduce the search space during the iterations of DP. Recall that $s_p$ denotes the length of the $p^{th}$ shortest arc in the road network $G$. 

\begin{proposition}
\label{prop:lowerbound_outgoing}
Given an outgoing path $P$ that contains $k$ vertices in $V_c$, let $w_P$ be the number of vertices without power service after the crew travels through $P$. Then the total service disruption time of any Hamiltonian cycle extended from $P$ is lower bounded by
\[
u^F(P) + w_Ps_1 + \sum_{p=2}^{n-k} (n-k+1-p)s_p.
\]
\end{proposition}

\begin{proposition}
\label{prop:lowerbound_return}
Given a return path $P$ that contains $k$ vertices in $V_c$, the total service disruption time of any Hamiltonian cycle extended from $P$ is lower bounded by
\[
\sum_{p=1}^{n-k+1}(n+1-p) s_p + v^B(P).
\]
\end{proposition}
%
%Please see the Appendix for the proofs of the two propositions above.
The proofs of the propositions above are in Appendices \ref{App:lowerbound_outgoing} and \ref{App:lowerbound_return} respectively.

\subsubsection{Upper bounds from heuristics}
\label{Sec:primal_heuristic}
To create initial solutions and upper bounds for optimal solutions, we use two greedy heuristics. The heuristics use different criteria based on the distance and the priority of vertices on the power distribution network. The heuristics work as follow.
\begin{itemize}
%    \item \emph{Greedy in priority (GiP):} Always go to the unvisited location with the highest priority. We determine the priority of a vertex based on the number of its successors on the distribution network. Accordingly, the power source has the highest priority (recall that $|S_1|=n$), while the leaf nodes have the lowest priority. 
    \item \emph{Greedy in distance (GiD):} Always go to the nearest unvisited location.
    \item \emph{Greedy in priority-weighted distance (GiPD):} Given the current location $i$, go the the unvisited vertex $j$ with the smallest $d_{ij}/|S_j|$ value.
\end{itemize}

\subsection{Bidirectional dynamic programming}
\label{Sec:BiDP}
 We present our bi-directional DP-based solution method in Algorithm \ref{alg:bidirection}. It starts with calculating initial feasible solutions and upper bounds, and the maximum position for each vertex in a Hamiltonian cycle. Then we perform the backward and forward DP simultaneously. For each outgoing path or return path, we calculate a lower bound following the ideas in Propositions \ref{prop:lowerbound_outgoing} and \ref{prop:lowerbound_return} to determine whether to accept this path or not. We also update the upper bound periodically by completing the outgoing paths greedily.

\begin{algorithm}[H]
\caption{A bidirectional dynamic programming algorithm for the \PRTRP{}}
\label{alg:bidirection}
\begin{algorithmic}[1]
\Require The road network $G=(V, A)$, power network $T=(V_c, A_c)$, travel time $d_{ij}$ for $i,j \in V$.
\Ensure A Hamiltonian cycle with the minimum total service disruption time.
\State (Pre-processing) For each $i \in V_c$, calculate the maximum position $\beta_i$ for vertex $i$ in the optimal Hamiltonian cycle. Compute an upper bound $U$ with the heuristics discussed in Section~\ref{Sec:primal_heuristic}.
\State (Initialization) Set $\Gamma^F_0 = \Gamma^B_0=\{(0)\}$ and $\Gamma^F_l = \Gamma^B_l= \emptyset$ for $l=1,\ldots, \lceil\frac{n}{2}\rceil+1$. 
%\State Computer upper bound $U$ using greedy heuristics
\For {$l=0$ to $\lceil\frac{n}{2} \rceil$}  \Comment{Path extension}
	\For {each return path $P$ in $\Gamma^B_l$}
		\For {each vertex $i \in V_c \setminus V(P)$ such that $n-l \le \beta_i$} \label{alg:bid:backext}
			\State Generate a path $P' = \{i\} \oplus P$. 
			%Compute a lower bound $\underline{t}^b_{P'}$. If $\underline{t}^b_{P'} \ge \overline{t}$, go to Step~\ref{alg:bid:backext}.
			\State Compute $v^B(P')$.
			\State Compute a lower bound $LB(P')$ for the return path $P'$ according to Proposition~\ref{prop:lowerbound_return}.
			\If{$U\geq LB(P')$ }
			\State $\Gamma^B_{l+1} \gets \texttt{BackwardUpdate}(\Gamma^B_{l+1}, P')$.
			\EndIf
		\EndFor
	\EndFor
	\For {each outgoing path $P$ in $\Gamma^F_l$}
		\For {each vertex $i \in V_c \setminus V(P)$ such that $l \le \beta_i$} \label{alg:bid:forext}
			\State Generate a path $P' = P \oplus \{i\}$. %Compute a lower bound $\underline{t}^f_{P'}$. If $\underline{t}^f_{P'} \ge \overline{t}$, go to Step~\ref{alg:bid:forext}.
			\State Compute $u^F(P')$.
			\State Compute a lower bound $LB(P')$ for the outgoing path $P'$ according to Proposition~\ref{prop:lowerbound_outgoing}.
			\If{$U\geq LB(P')$ }
			\State $\Gamma^F_{l+1} \gets \texttt{ForwardUpdate}(\Gamma^F_{l+1}, P')$.
			\EndIf
		\EndFor
	\EndFor
	\State Update upper bound $U$ 
\EndFor
\For{each path $P$ in $\Gamma^B_{\lceil\frac{n}{2}\rceil}$ (if $n$ is odd) or $\Gamma^B_{\lceil\frac{n}{2}\rceil+1}$ (if $n$ is even)} %\Comment{Termination}
	\For{each path $P'$ in $\Gamma^F_{\lceil\frac{n}{2} \rceil}$}
		\If{$P' \oplus P$ is a Hamiltonian cycle and $u^F(P') + v^B(P) < U$}
		  \State $U \gets u^F(P') + v^B(P)$ and store $P' \oplus P$.
		\EndIf
	\EndFor
\EndFor
\State \Return $U$ and the associated Hamiltonian cycle.
\end{algorithmic}
\end{algorithm}

\pagebreak
The subroutines for $\texttt{BackwardUpdate}(S, P)$ and $\texttt{ForwardUpdate}(S, P)$ work as follows.
\begin{algorithm}[h]
\caption{The subroutine $\texttt{BackwardUpdate}(S, P)$}
\begin{algorithmic}[1]
\Require A set $S$ of return paths and a return path $P$.
\Ensure An updated set $S$.
\If {there is no return path $P' \in S$ with the same configuration as $P$}
\State Add $P$ to $S$
\Else
\State Let path $P' \in S$ be the return path with the same configuration as $P$
    \If{$v^B(P) < v^B(P')$}
    \State Delete $P'$ from $S$ and add $P$ into $S$
    \EndIf
\EndIf
%\State Insert $P$ into $S$.
\State Return $S$.
\end{algorithmic}
\end{algorithm}

\begin{algorithm}[h]
\caption{The subroutine $\texttt{ForwardUpdate}(S, P)$}
\label{Alg:ForSub}
\begin{algorithmic}[1]
\Require A set $S$ of outgoing paths and an outgoing path $P$.
\Ensure An updated set $S$.
\If {there is no outgoing path $P' \in S$ with the same configuration as $P$}
\State Add $P$ to $S$
\Else
\State Let path $P' \in S$ be the outgoing path with the same configuration as $P$
    \If{$u^F(P) < u^F(P')$}
    \State Delete $P'$ from $S$ and add $P$ into $S$
    \EndIf
\EndIf
%\State Insert $P$ into $S$.
\State Return $S$.
\end{algorithmic}
\end{algorithm}

%In the following section, we present our computational results.

%\subsection{Special cases}
%line power network - early proposition
\section{Computational Experiments}
\label{Sec:Comp}
In this section, we present computational results of our solution method and compare it with benchmarks. The test instances are created from the Electric Power Research Institute (EPRI) instances, which are based on a U.S. power distribution network \cite{CKT5}. We use the $ckt5$ instance as shown in Figure \ref{Fig:ckt5}. The instance $ckt5$ represents the power distribution network, and we assume the road network is a complete graph. To create test instances over a power network of a smaller size, we take subtrees from the original power network. In our naming convention, $ckt5\_x$ represents the subtree of the original power network rooted at node $x$, i.e., node $x$ becomes the power source of the new instance. We assume that a unit distance can be traveled in a second, and this can be adjusted for the speed of travel.
\begin{figure}[h]
\centering
\includegraphics[scale=0.5]{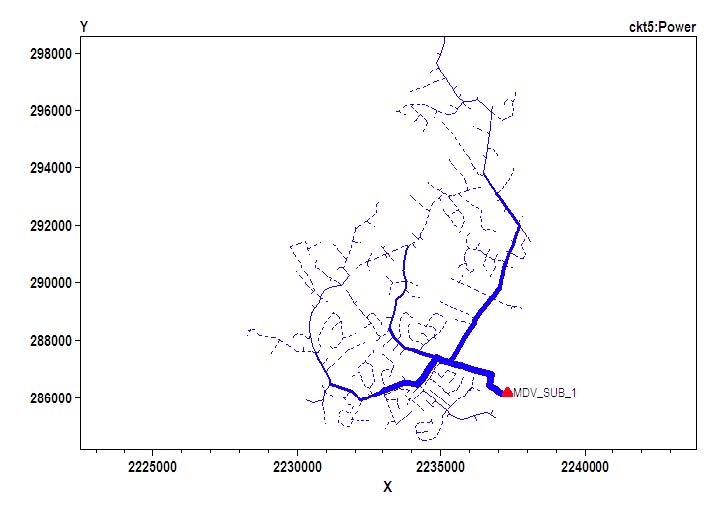}
\caption{The $ckt5$ instance by EPRI.}
\label{Fig:ckt5}
\end{figure}

We compare the optimal results obtained by our BiDP with the benchmark heuristics GiD and GiPD introduced in Section~\ref{Sec:primal_heuristic}. We omit the results obtained by the commercial optimization solver CPLEX since only very small-sized instances can be solved to optimality within the given time limit. All algorithms are implemented in the C programming language, and the experiments are performed on a computer that has a 2.8 GHz Intel i7 CPU with 4 cores and 16 GB of RAM and runs the macOS operating system. In the computational results, we report the resulting total service disruption time ($z$) and the computation time ($t$) in seconds for each algorithm. %\textcolor{red}{(Qie: is the unit of total service disruption time also second in the table? The optimal objective seems very small in seconds. How is the travel time between each pair of vertices given in the input?)}   
We omit the computation time for the benchmark heuristics since they are very quick to find a solution (in less than a second). 

%\textcolor{red}{(Qie: how fast, within a millisecond, a second, or several seconds?)}.

%In the computational experiments, we have different settings for the implementation of the bi-directional dynamic programming algorithm. In the first setting, we implement the algorithm as described earlier (Algorithms \ref{alg:bidirection}-\ref{Alg:ForSub}) and find the optimal solution. In other settings, we implement it heuristically to restrict the search space. 

\subsection*{Exact solutions by BiDP}
We present solutions found by our BiDP algorithm and benchmark heuristics in Table \ref{Tab:OptRes}.  Our BiDP algorithm finds optimal solutions for instances with fewer than 20 vertices in less than 165 seconds. GiD and GiPD deviate from the optimal objective by 27\% and 22\% on average, while the maximum deviation goes up to 50\% and 62\% respectively. The BiDP shows a clear advantage over the two heuristics in terms of solution quality. This could be weighed against the computation time in practice.

\begin{table}[H]
\caption{Solutions found by BiDP and benchmark heuristics ($z$ is in millions).} % title of Table
\centering 
\scalefont{0.8}
\begin{tabular}{l r r r c r}
\hline\hline 
Instance &$n$ & $z_{GiD}$ & $z_{GiPD}$ & $z_{BiDP}(opt)$ & $t_{BiDP}(opt)$ \\ [0.5ex]
\hline 
$ckt5\_5$	&	13	&	26.30	&	29.63	&	24.81	&	0.84	\\
$ckt5\_630$	&	13	&	27.58	&	27.30	&	26.55	&	1.26	\\
$ckt5\_520$	&	13	&	26.95	&	26.40	&	22.34	&	1.28	\\
$ckt5\_445$	&	15	&	27.90	&	33.02	&	23.00	&	2.86	\\
$ckt5\_513$	&	15	&	36.18	&	32.54	&	28.87	&	3.61	\\
$ckt5\_87$	&	15	&	34.82	&	38.87	&	30.88	&	6.00	\\
$ckt5\_135$	&	15	&	34.34	&	36.05	&	27.97	&	5.36	\\
$ckt5\_375$	&	17	&	52.51	&	48.33	&	40.21	&	45.37	\\
$ckt5\_144$	&	18	&	53.00	&	42.01	&	39.87	&	18.18	\\
$ckt5\_559$	&	18	&	51.36	&	37.84	&	36.28	&	118.41	\\
$ckt5\_820$	&	18	&	50.88	&	58.01	&	35.80	&	164.15	\\
$ckt5\_376$	&	19	&	49.15	&	44.18	&	40.27	&	97.62	\\
$ckt5\_299$	&	20	&	53.28	&	50.50	&	38.84	&	1647.36	\\
$ckt5\_361$	&	21	&	44.48	&	54.87	&	38.10	&	199.69	\\
$ckt5\_959$	&	21	&	67.70	&	62.33	&	50.88	&	715.72	\\
$ckt5\_732$	&	21	&	46.26	&	46.23	&	40.14	&	704.39	\\
$ckt5\_805$	&	21	&	59.14	&	51.67	&	50.57	&	3459.41	\\
$ckt5\_477$	&	22	&	63.83	&	57.51	&	45.81	&	1850.63	\\
$ckt5\_742$	&	22	&	78.63	&	78.37	&	53.89	&	2947.84	\\
$ckt5\_285$	&	23	&	66.99	&	55.93	&	44.72	&	4609.57	\\
$ckt5\_41$	&	23	&	61.70	&	54.54	&	50.05	&	15020.91	\\\hline
\multicolumn{2}{l}{Avg. Deviation} & 27\% & 22\% & &\\
\multicolumn{2}{l}{Min. Deviation} & 4\% & 2\% & &\\
\multicolumn{2}{l}{Max. Deviation} & 50\% & 62\% & &\\[0.5ex] % [1ex] adds vertical space
\hline \hline
\end{tabular}
\label{Tab:OptRes} 
\end{table}

We also observe that the computation times for instances with the same number of vertices can be significantly different. Consider instances $ckt5\_361$ and $ckt5\_805$. They both have 21 vertices. The optimal solution for $ckt5\_361$ is found in 199.69 seconds while it takes 3459.41 seconds for $ckt5\_805$. Similarly, $ckt5\_144$ and $ckt5\_820$ both have 18 vertices, and computation times are 18.18 and 164.15 seconds respectively. The main factor behind these differences is the structure of the power distribution network in the instances, as shown in Figures \ref{Fig:18} and \ref{Fig:21}. As more vertices branch out (i.e., having multiple immediate successors) and the vertices that branch out are closer to the power source, the computation time of our algorithm significantly increases. This is mainly due to the fact that the upper bounds in instances with fewer vertices branching out are tighter, which keeps the size of the search tree of BiDP small. 

\begin{figure}[H]
\centering
\subfigure[$ckt5\_144$]{%
\includegraphics[height=2in]{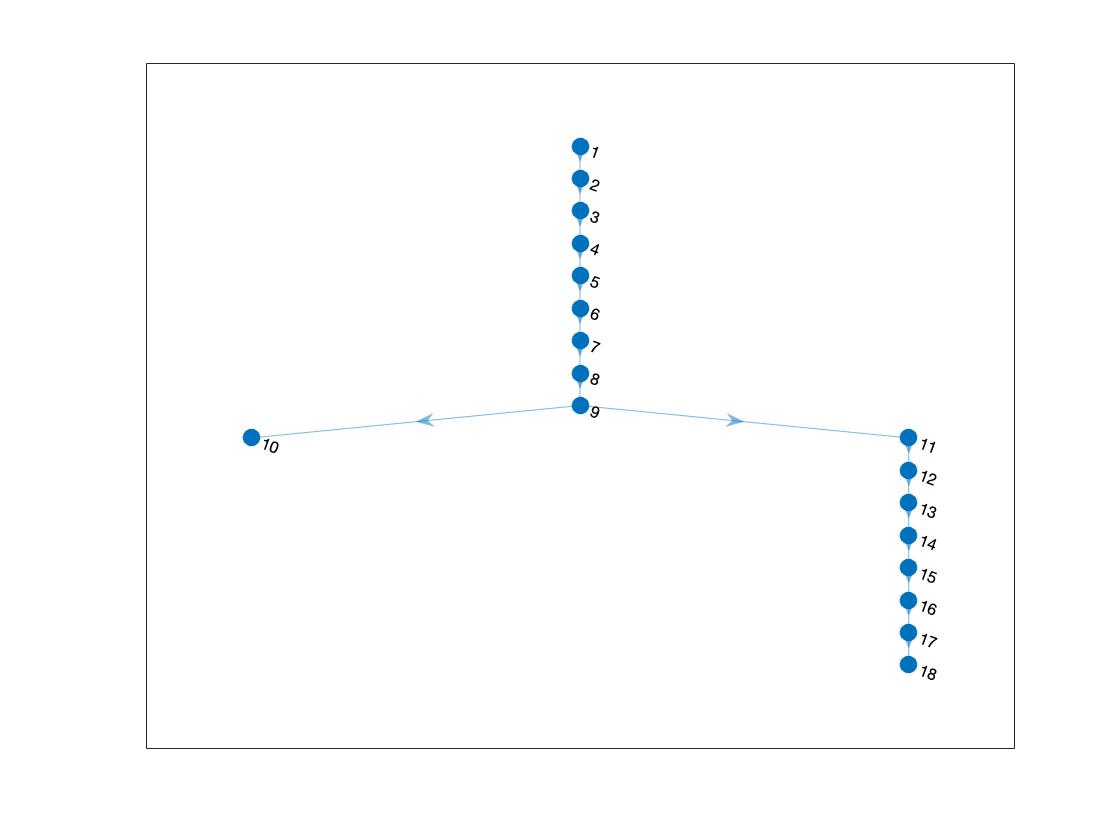}}%
\qquad
\subfigure[$ckt5\_820$]{%
\includegraphics[height=2in]{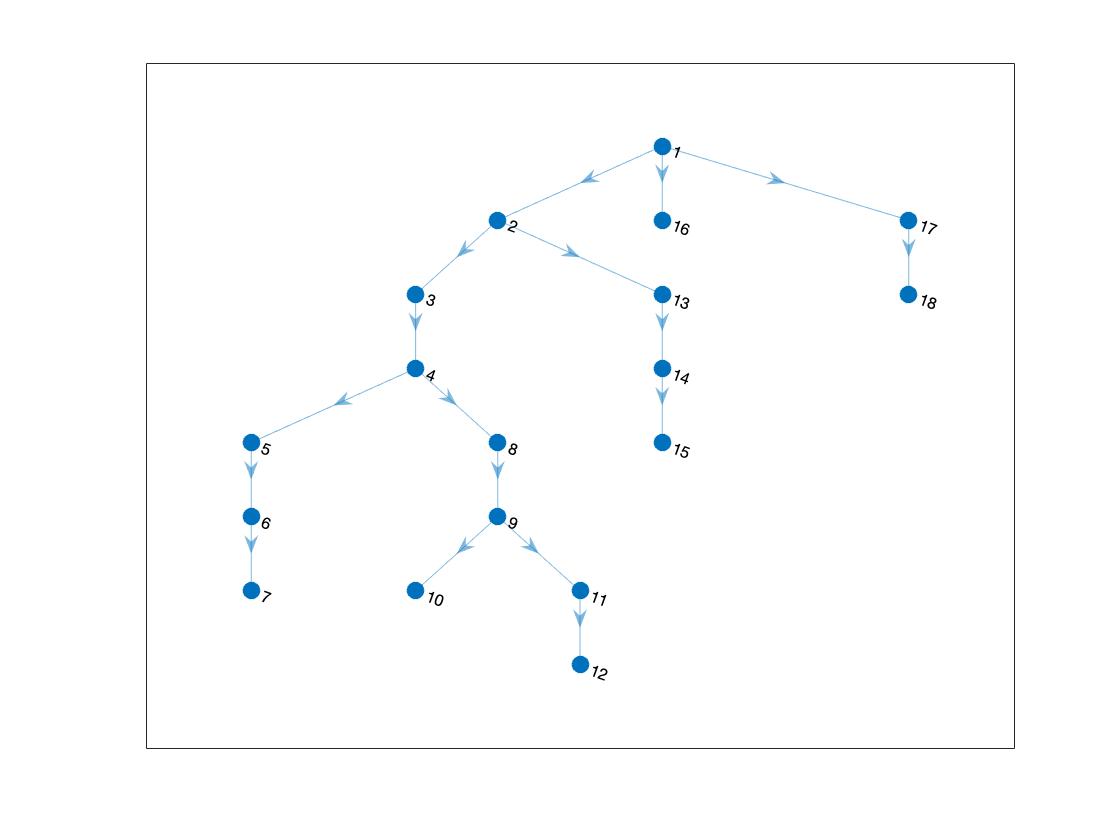}}%
\caption{Power distribution network for instances with 18 vertices.}
\label{Fig:18}
\end{figure}

\begin{figure}[H]
\centering
\subfigure[$ckt5\_361$]{%
\includegraphics[height=2in]{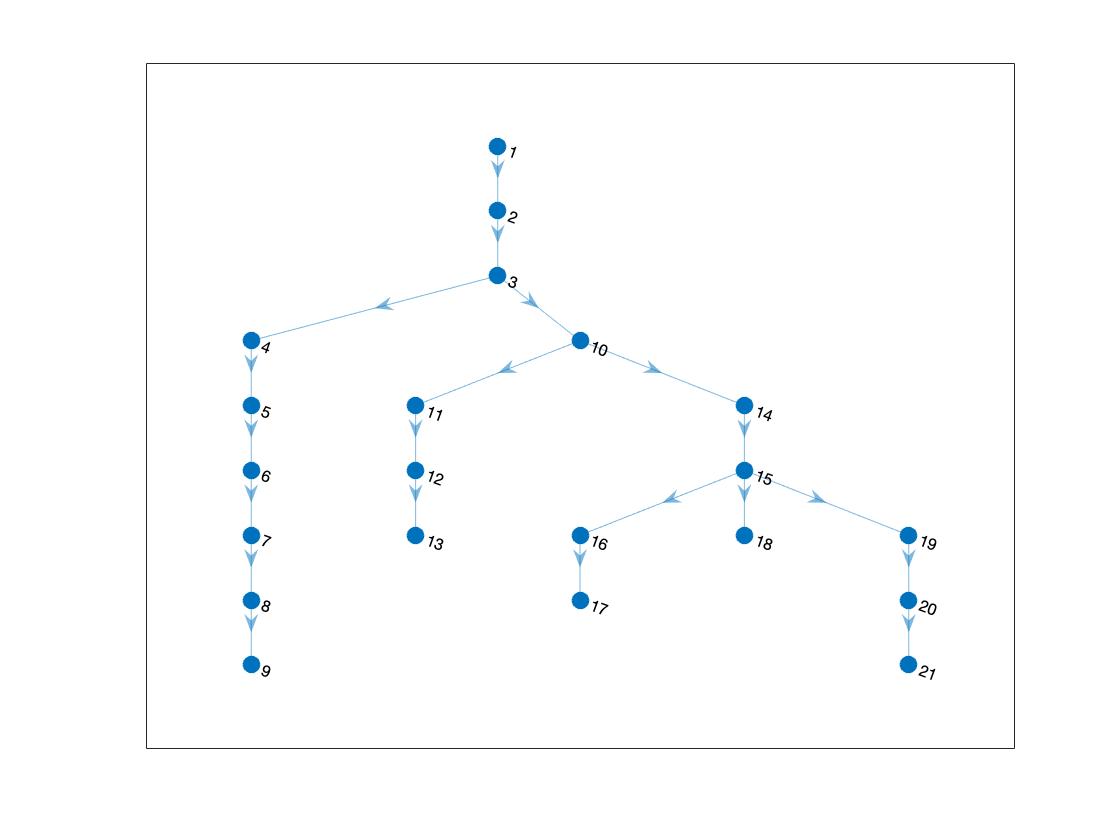}}%
\qquad
\subfigure[$ckt5\_805$]{%
\includegraphics[height=2in]{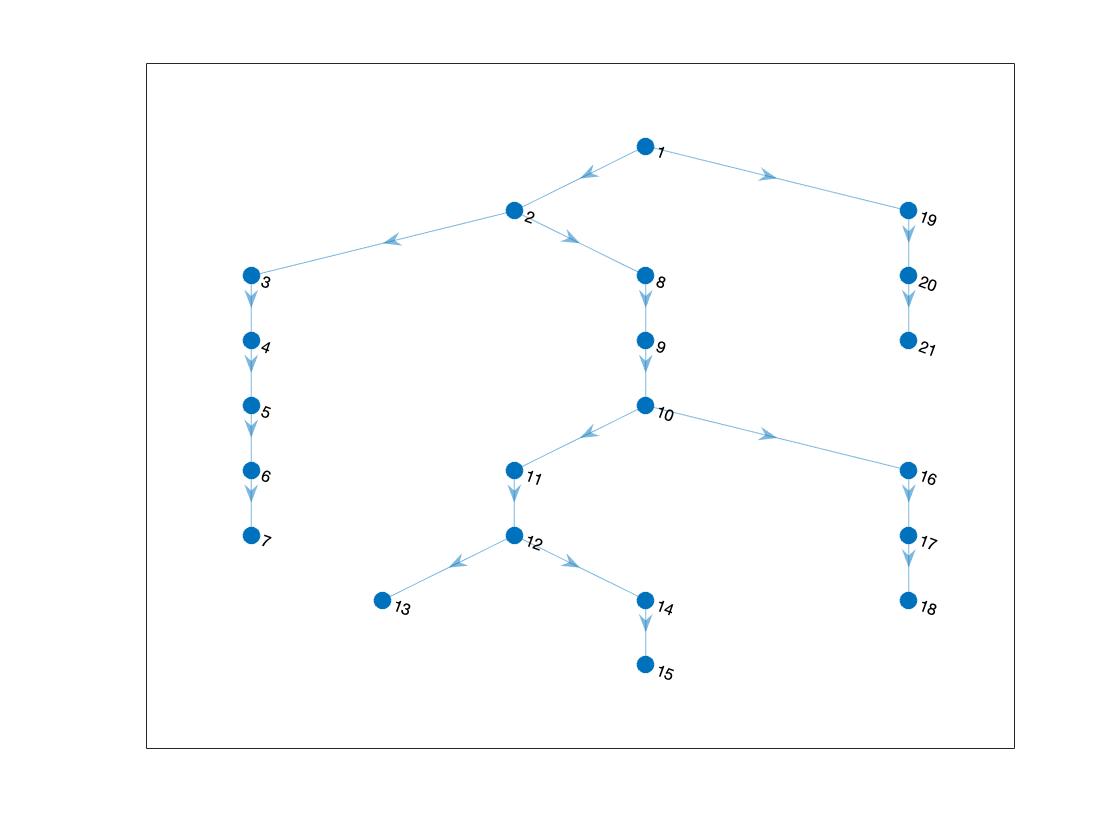}}%
\caption{Power distribution network for instances with 21 vertices.}
\label{Fig:21}
\end{figure}

\subsection*{Heuristic Implementations}
Motivated by the previous observation, we also implement our BiDP algorithm heuristically. In the first setting, we modify the use of the upper bounds. Recall that we initialize the algorithm with an upper bound that is generated by the greedy heuristics, and update the upper bound periodically as we progress with the bi-directional search. The quality of the upper bound improves over time. Therefore, we adjust the upper bound with coefficients $\theta$ and $\delta$. Accordingly, we allow the outgoing and return paths into our search tree if the corresponding lower bound is less than or equal to $(\theta + l \delta)U$, where $l$ is the number of vertices that have been added to the outgoing or return path. Therefore, we start with an aggressive use of the upper bound and relax it as we progress, creating a dynamic heuristic upper bound. For example, for $\theta=0.80$ and $\delta=0.01$, the effective upper bound when the first vertex is added to the outgoing or return path is $0.80U$, and it is $0.87U$ when adding the $8^{th}$ vertex. We report the computational results for $(\theta=0.80, \delta=0.01)$, $(\theta=0.90, \delta=0.01)$, and $(\theta=1.00, \delta=0.00)$ in Table \ref{Tab:HerBiDP1}. Note that the last one corresponds to the exact solution method BiDP. We also report the average, minimum and maximum optimality gap (i.e., percentage deviation from the optimal solution), and the average, minimum and maximum percentage reduction in computation time compared to $t(opt)$ for each ($\theta$, $\delta$) pair. We highlight the cases where we find optimal solutions with the heuristic upper bounds. In the computational results, the average optimality gap is 0.00\% for each ($\theta$, $\delta$) pair. Using $(\theta=0.80, \delta=0.01)$, we find optimal solutions in 19 instances out of 21. In the remaining instances, the deviations from optimal solutions are less than or equal to 1\%. The average percentage reduction in computation time is 68\%, and the maximum reduction is 90\%. Using $(\theta=0.90, \delta=0.01)$, we find the optimal solution in all but one instance with 11\% average reduction in computation time.
%While $\theta=0.90$ does not affect the computation time much, $\theta=0.80$ decreases the computation time significantly.  On average, optimal solutions are found by these heuristic implementations with 68\% reduction in computation time, and the maximum reduction is 82\%. 

\begin{table}[H]
\caption{Solutions found by heuristic implementation of BiDP with dynamic upper bounds, optimal solutions are highlighted ($z$ is in millions). } % title of Table
\centering 
\scalefont{0.8}
\begin{tabular}{l |c| r r| r r | r r}
\hline\hline 
& & \multicolumn{2}{|c|}{$\theta$=0.80, $\delta=0.01$} & \multicolumn{2}{|c}{$\theta$=0.90, $\delta=0.01$} & \multicolumn{2}{|c}{$\theta$=1.00, $\delta=0$}  \\ \hline
Instance &$n$ & $z$ & $t(sec)$ & $z$ & $t(sec)$  & $z(opt)$ & $t(sec)$ \\ [0.5ex]
\hline 
$ckt5\_5$	&	13	&	\cellcolor{blue!10}24.81	&	0.32	&	\cellcolor{blue!10}24.81	&	0.66	&	24.81	&	0.84	\\
$ckt5\_630$	&	13	&	\cellcolor{blue!10}26.55	&	0.78	&	\cellcolor{blue!10}26.55	&	1.09	&	26.55	&	1.26	\\
$ckt5\_520$	&	13	&	22.60	&	1.00	&	\cellcolor{blue!10}22.34	&	1.27	&	22.34	&	1.28	\\
$ckt5\_445$	&	15	&	\cellcolor{blue!10}23.00	&	1.39	&	\cellcolor{blue!10}23.00	&	2.32	&	23.00	&	2.86	\\
$ckt5\_513$	&	15	&	\cellcolor{blue!10}28.87	&	1.50	&	\cellcolor{blue!10}28.87	&	2.83	&	28.87	&	3.61	\\
$ckt5\_87$	&	15	&	\cellcolor{blue!10}30.88	&	2.06	&	\cellcolor{blue!10}30.88	&	4.96	&	30.88	&	6.00	\\
$ckt5\_135$	&	15	&	\cellcolor{blue!10}27.97	&	2.43	&	\cellcolor{blue!10}27.97	&	4.47	&	27.97	&	5.36	\\
$ckt5\_375$	&	17	&	\cellcolor{blue!10}40.21	&	15.58	&	\cellcolor{blue!10}40.21	&	39.57	&	40.21	&	45.37	\\
$ckt5\_144$	&	18	&	39.90	&	2.90	&	39.90	&	13.19	&	39.87	&	18.18	\\
$ckt5\_559$	&	18	&	\cellcolor{blue!10}36.28	&	47.71	&	\cellcolor{blue!10}36.28	&	106.89	&	36.28	&	118.41	\\
$ckt5\_820$	&	18	&	\cellcolor{blue!10}35.80	&	50.42	&	\cellcolor{blue!10}35.80	&	148.24	&	35.80	&	164.15	\\
$ckt5\_376$	&	19	&	\cellcolor{blue!10}40.27	&	13.13	&	\cellcolor{blue!10}40.27	&	82.47	&	40.27	&	97.62	\\
$ckt5\_299$	&	20	&	\cellcolor{blue!10}38.84	&	699.58	&	\cellcolor{blue!10}38.84	&	1655.45	&	38.84	&	1647.36	\\
$ckt5\_361$	&	21	&	\cellcolor{blue!10}38.10	&	46.95	&	\cellcolor{blue!10}38.10	&	183.89	&	38.10	&	199.69	\\
$ckt5\_959$	&	21	&   \cellcolor{blue!10}50.88	&	127.48	&	\cellcolor{blue!10}50.88	&	697.30	&	50.88	&	715.72	\\
$ckt5\_732$	&	21	&	\cellcolor{blue!10}40.14	&	214.50	&   \cellcolor{blue!10}	40.14	&	706.06	&	40.14	&	704.39	\\
$ckt5\_805$	&	21	&	\cellcolor{blue!10}50.57	&	812.12	&	\cellcolor{blue!10}50.57	&	3497.89	&	50.57	&	3459.41	\\
$ckt5\_477$	&	22	&	\cellcolor{blue!10}45.81	&	285.61	&	\cellcolor{blue!10}45.81	&	2044.89	&	45.81	&	1850.63	\\
$ckt5\_742$	&	22	&	\cellcolor{blue!10}53.89	&	407.28	&	\cellcolor{blue!10}53.89	&	3068.79	&	53.89	&	2947.84	\\
$ckt5\_285$	&	23	&	\cellcolor{blue!10}44.72	&	463.82	&	\cellcolor{blue!10}44.72	&	3684.37	&	44.72	&	4609.57	\\
$ckt5\_41$	&	23	&	\cellcolor{blue!10}50.05	&	1490.81	&	\cellcolor{blue!10}50.05	&	11591.40	&	50.05	&	15020.91	\\% [1ex] adds vertical space
\hline\hline 
\multicolumn{2}{l|}{Avg. Gap}  & 0.00\% && 0.00\% & & & \\ \hline
 \multicolumn{2}{l|}{Avg. Reduction in $t$ }  & & 68\% && 11\% & &\\
 \multicolumn{2}{l|}{Max Reduction in $t$ }  & & 90\% && 27\% & &\\
 \multicolumn{2}{l|}{Min Reduction in $t$ }  & & 22\% && -10\% & &
 \\
\hline \hline
\end{tabular}
\label{Tab:HerBiDP1} 
\end{table}

In the next heuristic implementation, in addition to the modified upper bounds, we also use a modified maximum position for the power source. After finding greedy solutions with GiP and GiPD, we check the positions of the power source in both solutions, and use the greater position to determine the maximum position for the power source during BiDP. 
%we allow the one that visits (and repairs) the power source later than the other one to determine the threshold for the power source. 
%\textcolor{red}{(Qie: I'm not sure if I understand this sentence. How is the new threshold calculated?)}
Table \ref{Tab:HerBiDP2} presents the corresponding solutions. We repeat the optimal solutions in the table for completeness. Optimal solutions found by the heuristic approaches are highlighted. Using this heuristic implementation, we can still find optimal solutions in almost all instances with 0.00\% average optimality gap. The results corresponding to $(\theta=1.00, \delta=0.00)$ show the impact of the heuristic maximum position only. We observe that by just using the heuristic maximum position, computation times to find optimal solutions decrease by 22\% on average and can be reduced up to 46\%. When heuristic upper bounds for $(\theta=0.80, \delta=0.01)$ are coupled with heuristic maximum position, the computation times to find optimal solutions decrease by 72\% on average and can be reduced up to 93\%.
\begin{table}[H]
\caption{Solutions found by heuristic implementation of BiDP with dynamic upper bounds and a heuristic maximum position, optimal solutions are highlighted ($z$ is in millions). } % title of Table
\centering 
\scalefont{0.8}
\begin{tabular}{l |c| r r| r r| r r|r r}
\hline\hline 
& & \multicolumn{2}{|c|}{$\theta$=0.80, $\delta=0.01$} & \multicolumn{2}{|c|}{$\theta$=0.90, $\delta=0.01$} & \multicolumn{2}{|c|}{$\theta$=1.00, $\delta=0$}\\ \hline
Instance &$n$ & $z$ & $t(sec)$ & $z$ & $t(sec)$ & $z$ & $t(sec)$ & $z(opt)$ & $t(sec)$ \\ [0.5ex]
\hline 
$ckt5\_5$	&	13	&	\cellcolor{blue!10}24.81	&	0.32	&	\cellcolor{blue!10}24.81	&	0.52	&	\cellcolor{blue!10}24.81	&	0.70	&	24.81	&	0.84	\\
$ckt5\_630$	&	13	&	\cellcolor{blue!10}26.55	&	0.63	&	\cellcolor{blue!10}26.55	&	0.81	&	\cellcolor{blue!10}26.55	&	0.93	&	26.55	&	1.26	\\
$ckt5\_520$	&	13	&	22.60	&	0.87	&	\cellcolor{blue!10}22.34	&	1.03	&	\cellcolor{blue!10}22.34	&	1.12	&	22.34	&	1.28	\\
$ckt5\_445$	&	15	&	\cellcolor{blue!10}23.00	&	1.17	&	\cellcolor{blue!10}23.00	&	1.82	&	\cellcolor{blue!10}23.00	&	2.27	&	23.00	&	2.86	\\
$ckt5\_513$	&	15	&	\cellcolor{blue!10}28.87	&	1.26	&	\cellcolor{blue!10}28.87	&	2.25	&	\cellcolor{blue!10}28.87	&	2.80	&	28.87	&	3.61	\\
$ckt5\_87$	&	15	&	\cellcolor{blue!10}30.88	&	1.71	&	\cellcolor{blue!10}30.88	&	3.59	&	\cellcolor{blue!10}30.88	&	4.41	&	30.88	&	6.00	\\
$ckt5\_135$	&	15	&	\cellcolor{blue!10}27.97	&	2.47	&	\cellcolor{blue!10}27.97	&	4.45	&	\cellcolor{blue!10}27.97	&	5.28	&	27.97	&	5.36	\\
$ckt5\_375$	&	17	&	\cellcolor{blue!10}40.21	&	12.71	&	\cellcolor{blue!10}40.21	&	26.46	&	\cellcolor{blue!10}40.21	&	29.42	&	40.21	&	45.37	\\
$ckt5\_144$	&	18	&	39.90	&	3.16	&	39.90	&	12.61	&	\cellcolor{blue!10}39.87	&	13.88	&	39.87	&	18.18	\\
$ckt5\_559$	&	18	&	\cellcolor{blue!10}36.28	&	34.30	&	\cellcolor{blue!10}36.28	&	72.86	&	\cellcolor{blue!10}36.28	&	67.54	&	36.28	&	118.41	\\
$ckt5\_820$	&	18	&	\cellcolor{blue!10}35.80	&	54.15	&	\cellcolor{blue!10}35.80	&	148.68	&	\cellcolor{blue!10}35.80	&	163.64	&	35.80	&	164.15	\\
$ckt5\_376$	&	19	&	\cellcolor{blue!10}40.27	&	11.08	&	\cellcolor{blue!10}40.27	&	58.51	&	\cellcolor{blue!10}40.27	&	70.88	&	40.27	&	97.62	\\
$ckt5\_299$	&	20	&	39.87	&	692.24	&	39.87	&	1418.70	&	39.87	&	1402.11	&	38.84	&	1647.36	\\
$ckt5\_361$	&	21	&	\cellcolor{blue!10}38.10	&	44.64	&	\cellcolor{blue!10}38.10	&	149.04	&	\cellcolor{blue!10}38.10	&	156.77	&	38.10	&	199.69	\\
$ckt5\_959$	&	21	&	\cellcolor{blue!10}50.88	&	118.11	&	\cellcolor{blue!10}50.88	&	524.20	&	\cellcolor{blue!10}50.88	&	535.61	&	50.88	&	715.72	\\
$ckt5\_732$	&	21	&	\cellcolor{blue!10}40.14	&	181.21	&	\cellcolor{blue!10}40.14	&	494.20	&	\cellcolor{blue!10}40.14	&	495.83	&	40.14	&	704.39	\\
$ckt5\_805$	&	21	&	\cellcolor{blue!10}50.57	&	537.72	&	\cellcolor{blue!10}50.57	&	2197.78	&	\cellcolor{blue!10}50.57	&	1865.39	&	50.57	&	3459.41	\\
$ckt5\_477$	&	22	&	\cellcolor{blue!10}45.81	&	297.54	&	\cellcolor{blue!10}45.81	&	2195.59	&	\cellcolor{blue!10}45.81	&	1787.22	&	45.81	&	1850.63	\\
$ckt5\_742$	&	22	&	\cellcolor{blue!10}53.89	&	411.80	&	\cellcolor{blue!10}53.89	&	3126.82	&	\cellcolor{blue!10}53.89	&	2928.76	&	53.89	&	2947.84	\\
$ckt5\_285$	&	23	&	\cellcolor{blue!10}44.72	&	340.01	&	\cellcolor{blue!10}44.72	&	2469.44	&	\cellcolor{blue!10}44.72	&	3007.83	&	44.72	&	4609.57	\\
$ckt5\_41$	&	23	&	\cellcolor{blue!10}50.05	&	1087.21	&	\cellcolor{blue!10}50.05	&	7137.93	&	\cellcolor{blue!10}50.05	&	9076.16	&	50.05	&	15020.91	\\% [1ex] adds vertical space
\hline\hline 
\multicolumn{2}{l|}{Avg. Gap}  & 0.00\% && 0.00\% & & 0.00\%&  & &\\ \hline
 \multicolumn{2}{l|}{Avg. Reduction in $t$ }   & & 72\% && 28\% & & 22\%& &\\
  \multicolumn{2}{l|}{Max Reduction in $t$ }   & & 93\% && -19\% & & 0\%& &\\
   \multicolumn{2}{l|}{Min Reduction in $t$ }   & & 32\% && 52\% & & 46\%& &\\
\hline \hline
\end{tabular}
\label{Tab:HerBiDP2} 
\end{table}

\section{Conclusion}
\label{Sec:Conc}
In this paper, we study a routing problem arising in power service restoration, where the repair crew travels to a number of locations to repair power equipment with the goal of minimizing the total power service disruption time. The service disruption time of a location depends on both the travel time on the road network and the topology of the power network, which brings a new layer of complexity to this routing problem. We call this problem the \PRTRP{}. 

The \PRTRP{} is computationally challenging to solve even for small-sized instances. We propose a bi-directional dynamic programming (BiDP) method to find optimal solutions. We develop several techniques to further reduce the search space of the dynamic programming, including pre-computed bounds of the optimal position of a vertex, dynamic lower bounds of optimal objectives of extending a path to Hamiltonian cycle, and upper bounds by two primal heuristics. In addition, we propose heuristic implementations of our BiDP based on dynamic heuristic upper bounds and heuristic bounds for optimal vertex positions. We test our methods on instances created from EPRI instances. Our BiDP method is able to find optimal solutions of instances that cannot be solved by off-the-shelf optimization software. Our heuristic implementations further reduce the computation time significantly while finding optimal solutions for most instances. 

There are several directions to extend our model and methods for future research. Firstly, more computationally efficient algorithms will need to be investigated to support quick decisions under emergency response. Secondly, it's worth considering routing multiple repair crew as power utility companies usually dispatch a fleet of utility trucks during restoration. An added complexity is that the total service disruption time over a route could depend on other routes, while in traditional vehicle routing problems the total travel time or distance of a route only depends on the route itself. Thirdly, service restoration often faces uncertainty in practice, e.g., equipment status and road availability. Taking these uncertainties into account in the model will help generate more robust restoration solutions. Lastly, the problem structure studied in this paper is common in service restoration of other networks, such as communication network or gas network. It would be interesting to see our results to be applied to other infrastructure networks.

\bibliographystyle{plain}

\bibliography{restoration}

\begin{thebibliography}{10}

\bibitem{arab2015proactive}
Ali Arab, Amin Khodaei, Zhu Han, and Suresh~K Khator.
\newblock Proactive recovery of electric power assets for resiliency
  enhancement.
\newblock {\em IEEE Access}, 3:99--109, 2015.

\bibitem{arab2016electric}
Ali Arab, Amin Khodaei, Suresh~K Khator, and Zhu Han.
\newblock Electric power grid restoration considering disaster economics.
\newblock {\em IEEE Access}, 4:639--649, 2016.

\bibitem{archer2008faster}
Aaron Archer, Asaf Levin, and David~P Williamson.
\newblock A faster, better approximation algorithm for the minimum latency
  problem.
\newblock {\em SIAM Journal on Computing}, 37(5):1472--1498, 2008.

\bibitem{arif2018optimizing}
Anmar Arif, Shanshan Ma, Zhaoyu Wang, Jianhui Wang, Sarah~M Ryan, and Chen
  Chen.
\newblock Optimizing service restoration in distribution systems with uncertain
  repair time and demand.
\newblock {\em IEEE Transactions on Power Systems}, 33(6):6828--6838, 2018.

\bibitem{arif2017power}
Anmar Arif, Zhaoyu Wang, Jianhui Wang, and Chen Chen.
\newblock Power distribution system outage management with co-optimization of
  repairs, reconfiguration, and {DG} dispatch.
\newblock {\em IEEE Transactions on Smart Grid}, 9(5):4109--4118, 2017.

\bibitem{bienstock2007using}
Daniel Bienstock and Sara Mattia.
\newblock Using mixed-integer programming to solve power grid blackout
  problems.
\newblock {\em Discrete Optimization}, 4(1):115--141, 2007.

\bibitem{binato2001new}
Silvio Binato, M{\'a}rio Veiga~F Pereira, and S{\'e}rgio Granville.
\newblock A new {B}enders decomposition approach to solve power transmission
  network design problems.
\newblock {\em IEEE Transactions on Power Systems}, 16(2):235--240, 2001.

\bibitem{chen2018toward}
Bo~Chen, Zhigang Ye, Chen Chen, Jianhui Wang, Tao Ding, and Zhaohong Bie.
\newblock Toward a synthetic model for distribution system restoration and crew
  dispatch.
\newblock {\em IEEE Transactions on Power Systems}, 34(3):2228--2239, 2018.

\bibitem{coffrin2015transmission}
Carleton Coffrin and Pascal Van~Hentenryck.
\newblock Transmission system restoration with co-optimization of repairs, load
  pickups, and generation dispatch.
\newblock {\em International Journal of Electrical Power \& Energy Systems},
  72:144--154, 2015.

\bibitem{coffrin2012last}
Carleton Coffrin, Pascal Van~Hentenryck, and Russell Bent.
\newblock Last-mile restoration for multiple interdependent infrastructures.
\newblock In {\em Twenty-sixth AAAI Conference on Artificial Intelligence},
  2012.

\bibitem{gao2020optimal}
Xue Gao and Zhi Chen.
\newblock Optimal restoration strategy to enhance the resilience of
  transmission system under windstorms.
\newblock In {\em 2020 IEEE Texas Power and Energy Conference (TPEC)}, pages
  1--6. IEEE, 2020.

\bibitem{Goldbeck2020-si}
Nils Goldbeck, Panagiotis Angeloudis, and Washington Ochieng.
\newblock Optimal supply chain resilience with consideration of failure
  propagation and repair logistics.
\newblock {\em Transp. Res. Part E: Logist. Trans. Rev.}, 133:101830, January
  2020.

\bibitem{hardy1952inequalities}
Godfrey~Harold Hardy, John~Edensor Littlewood, and George P{\'o}lya.
\newblock {\em Inequalities}.
\newblock Cambridge University Press, 1952.

\bibitem{CKT5}
Electric Power~Research Institute.
\newblock {EPRI} test circuits.
\newblock \url{https://smartgrid.epri.com/SimulationTool.aspx}, Accessed March
  2022.

\bibitem{jufri2019state}
Fauzan~Hanif Jufri, Victor Widiputra, and Jaesung Jung.
\newblock State-of-the-art review on power grid resilience to extreme weather
  events: Definitions, frameworks, quantitative assessment methodologies, and
  enhancement strategies.
\newblock {\em Applied Energy}, 239:1049--1065, 2019.

\bibitem{luo2014branch}
Zhixing Luo, Hu~Qin, and Andrew Lim.
\newblock Branch-and-price-and-cut for the multiple traveling repairman problem
  with distance constraints.
\newblock {\em European Journal of Operational Research}, 234(1):49--60, 2014.

\bibitem{lysgaard2014branch}
Jens Lysgaard and Sanne W{\o}hlk.
\newblock A branch-and-cut-and-price algorithm for the cumulative capacitated
  vehicle routing problem.
\newblock {\em European Journal of Operational Research}, 236(3):800--810,
  2014.

\bibitem{morshedlou2018work}
Nazanin Morshedlou, Andr{\'e}s~D Gonz{\'a}lez, and Kash Barker.
\newblock Work crew routing problem for infrastructure network restoration.
\newblock {\em Transportation Research Part B: Methodological}, 118:66--89,
  2018.

\bibitem{ngueveu2010effective}
Sandra~Ulrich Ngueveu, Christian Prins, and Roberto~Wolfler Calvo.
\newblock An effective memetic algorithm for the cumulative capacitated vehicle
  routing problem.
\newblock {\em Computers \& Operations Research}, 37(11):1877--1885, 2010.

\bibitem{WP}
Washington Post.
\newblock After {I}rma, {F}lorida prepares for days — and maybe weeks —
  without power.
\newblock
  \url{https://www.washingtonpost.com/news/post-nation/wp/2017/09/12/florida-struggles-with-top-job-in-irmas-wake-restoring-power-to-millions/},
  September 2017.

\bibitem{prakash2016review}
Krishneel Prakash, Avneel Lallu, FR~Islam, and KA~Mamun.
\newblock Review of power system distribution network architecture.
\newblock In {\em 2016 3rd Asia-Pacific World Congress on Computer Science and
  Engineering (APWC on CSE)}, pages 124--130. IEEE, 2016.

\bibitem{QIURestoration}
Feng Qiu and Peijie Li.
\newblock An integrated approach for power system restoration planning.
\newblock {\em Proceedings of the IEEE}, 105(7):1234--1252, 2017.

\bibitem{NPR}
National~Public Radio.
\newblock Why it's so hard to turn the lights back on in {P}uerto {R}ico.
\newblock
  \url{https://www.npr.org/2017/10/20/558743790/why-its-so-hard-to-turn-the-lights-back-on-in-puerto},
  October 2017.

\bibitem{ribeiro2012adaptive}
Glaydston~Mattos Ribeiro and Gilbert Laporte.
\newblock An adaptive large neighborhood search heuristic for the cumulative
  capacitated vehicle routing problem.
\newblock {\em Computers \& Operations Research}, 39(3):728--735, 2012.

\bibitem{rivera2015multistart}
Juan~Carlos Rivera, H~Murat Afsar, and Christian Prins.
\newblock A multistart iterated local search for the multitrip cumulative
  capacitated vehicle routing problem.
\newblock {\em Computational Optimization and Applications}, 61(1):159--187,
  2015.

\bibitem{simon2012randomized}
Ben Simon, Carleton Coffrin, and Pascal Van~Hentenryck.
\newblock Randomized adaptive vehicle decomposition for large-scale power
  restoration.
\newblock In {\em International Conference on Integration of Artificial
  Intelligence (AI) and Operations Research (OR) Techniques in Constraint
  Programming}, pages 379--394. Springer, 2012.

\bibitem{Suryawanshi2022-ze}
Pravin Suryawanshi and Pankaj Dutta.
\newblock Optimization models for supply chains under risk, uncertainty, and
  resilience: A state-of-the-art review and future research directions.
\newblock {\em Transp. Res. Part E: Logist. Trans. Rev.}, 157:102553, January
  2022.

\bibitem{tan2019scheduling}
Yushi Tan, Feng Qiu, Arindam~K Das, Daniel~S Kirschen, Payman Arabshahi, and
  Jianhui Wang.
\newblock Scheduling post-disaster repairs in electricity distribution
  networks.
\newblock {\em IEEE Transactions on Power Systems}, 34(4):2611--2621, 2019.

\bibitem{thiebaux2013planning}
Sylvie Thi{\'e}baux, Carleton Coffrin, Hassan Hijazi, and John Slaney.
\newblock Planning with {MIP} for supply restoration in power distribution
  systems.
\newblock In {\em Twenty-Third International Joint Conference on Artificial
  Intelligence}, 2013.

\bibitem{Ulusan2021-su}
Aybike Ulusan and {\"O}zlem Ergun.
\newblock Approximate dynamic programming for network recovery problems with
  stochastic demand.
\newblock {\em Transp. Res. Part E: Logist. Trans. Rev.}, 151:102358, July
  2021.

\bibitem{van2015transmission}
Pascal Van~Hentenryck and Carleton Coffrin.
\newblock Transmission system repair and restoration.
\newblock {\em Mathematical Programming}, 151(1):347--373, 2015.

\bibitem{van2011vehicle}
Pascal Van~Hentenryck, Carleton Coffrin, Russell Bent, et~al.
\newblock Vehicle routing for the last mile of power system restoration.
\newblock In {\em Proceedings of the 17th Power Systems Computation Conference
  (PSCC’11), Stockholm, Sweden}. Citeseer, 2011.

\bibitem{wu2004exact}
Bang~Ye Wu, Zheng-Nan Huang, and Fu-Jie Zhan.
\newblock Exact algorithms for the minimum latency problem.
\newblock {\em Information Processing Letters}, 92(6):303--309, 2004.

\end{thebibliography}

\appendix

\section{The lower-bounding lemma}
We first introduce the Rearrangement Inequality and a lower-bounding Lemma that will be used in the proofs of Propositions~\ref{Pro:lower}, \ref{prop:lowerbound_outgoing}, and \ref{prop:lowerbound_return}. 
Given $x \in \Re^n$, let $f(x)$ denote the vector in $\Re^n$ with the components of $x$ rearranged in a non-increasing order and $g(x)$ denote the vector in $\Re^n$ with the components of $x$ rearranged in a non-decreasing order. For example, given $x = (3, 2, 4, 1)$, we have $f(x) = (4, 3, 2, 1)$ and $g(x) = (1, 2, 3, 4)$. Given two vectors $x, y \in \Re^n$, the Rearrangement Inequality~\cite{hardy1952inequalities} states that
\begin{equation} \label{eq:rearrangement}
x \cdot y \ge f(x) \cdot g(y),
\end{equation}
where $x \cdot y = \sum_{i \in [1:n]} x_iy_i$. Given two vectors $x, y \in \Re^n$, we use $x \ge y$ to denote that $x_i \ge y_i$ for $i \in [1:n]$, i.e., vector $x$ is no smaller than vector $y$ component-wise.

We introduce a different way of calculating the total disruption time in the \PRTRP{}. Given a Hamiltonian cycle $H = (j_0=0, j_1, \ldots, j_n, 0)$ over the graph $G$ with $j_1, \ldots, j_n \in V_c$, let $T(H)$ denote the total service disruption time if the repair crew follows $H$. 
Let $c_p(H)$ be the number of vertices with no power right before the repair crew arrives at vertex $j_p$ following $H$, for $p \in [1:n]$. Then we have the following result.
\begin{equation} \label{eq:disruption_time}
    T(H) = \sum_{p \in [1:n]} c_p(H) d_{j_{p-1}j_p}. 
\end{equation}
Equation~\eqref{eq:disruption_time} simply states that the total disruption time following $H$ is to sum up the total service disruption time between visiting two consecutive vertices $j_{p-1}$ and $j_p$ on $H$, which is the product of the number of vertices with no power service $c_p$ and the travel time $d_{j_{p-1}j_p}$.

Let $C(H)=(c_1(H), c_2(H), \ldots, c_n(H))$ and $D(H) = (d_{0j_1}, d_{j_1j_2}, \ldots, d_{j_{n-1}j_n})$. We now introduce the lower-bounding lemma.
\begin{lemma} \label{lemma:lowerbounding}
Given a Hamiltonian cycle $H$ for the \PRTRP{},
\begin{equation} \label{eq:disruption_time:lower_bound}
T(H) \ge C(H) \cdot g(D(H)).
\end{equation}
\end{lemma}

\begin{proof}
We have
\[
T(H) = C(H) \cdot D(H) \ge f(C(H)) \cdot g(D(H)) = C(H) \cdot g(D(H)),
\]
where the first equality follows from Equation~\eqref{eq:disruption_time}, the first inequality follows from the Rearrangement Inequality~\eqref{eq:rearrangement}, and the second equality follows from the fact that the components of $C$ are already sorted in a non-increasing order. 
\end{proof}

\section{Proof of Proposition \ref{Pro:lower}} \label{App:lower}
\begin{proof}[Proof of Proposition \ref{Pro:lower}]
According to Lemma~\ref{lemma:lowerbounding}, for any Hamiltonian cycle $H$, we have
\[T(H) \ge C(H) \cdot g(D(H)).\]
We will find vectors $\bar{C}$ and $\bar{D}$ such that $C(H) \ge \bar{C}$ and $g(D(H)) \ge \bar{D}$ respectively. 

%In addition, $c_p \le n-(p-1)$ for $p \in [1:n]$, since we cannot have more than $p-1$ repaired vertices before arriving at $j_p$.

We first construct a lower bound of the vector $C(H)$. Consider any Hamiltonian cycle where vertex $i$ is visited at position $k$ with $k > n-|S_i|$, i.e., $(j_0=0, j_1, \dots, j_{k-1}, i, j_{k+1}, \dots, j_n)$. Since $k>n-|S_i|$, the positions on this Hamiltonian cycle can be divided into three groups as follows: 
\begin{enumerate}
\item[(i)] $p\in [1:n-|S_i|]$, 
\item[(ii)] $p \in [n-|S_i|+1:k]$,
\item[(iii)] $p\in [k+1:n]$.
\end{enumerate}
For the positions (i) $p\in [1:n-|S_i|]$, with each repair, at most one additional vertex can be recovered. This is possible since vertices can be selected from the set $V\setminus|S_i|$, and the repair order can follow the precedence relation on the power network. For positions (ii) $p \in [n-|S_i|+1:k]$, no additional vertices can receive service because all remaining vertices are successors of vertex $i$ or vertex $i$ itself. Finally, for positions (iii) $p \in [k+1,n]$, one additional vertex can be recovered, similar to (i). Accordingly, the total number of vertices that do not have service for each arc on the HC corresponds to a vector $\bar{C}=(c_1=n, c_2=n-1, n-2, \dots, |S_i|+1, |S_i|, \dots, |S_i|,|S_i|-1, |S_i|-2, \dots, c_n=1 )$. Vector $\bar{C}$ represents the best-case scenario for reducing the number of vertices without power service given that vertex $i$ is visited at position $k$. This implies that $C(H) \ge \bar{C}$.

Now we construct a lower bound of the vector $g(D(H))$. Recall that $s_p$ is the length of the $p^{th}$ shortest arc in $A$. Consider the vector $\bar{D}=(s_1, s_2, \dots, s_n)$. Since each component of $g(D(H))$ is the length of some arc in $A$ and the components of $g(D(H))$ is sorted in a non-increasing order. We have $g(D(H)) \ge \bar{D}$. Combining the facts that $C(H) \ge \bar{C}$ and $g(D(H)) \ge \bar{D}$ and~\eqref{eq:disruption_time:lower_bound}, we have
\[T(H) \ge \bar{C} \cdot \bar{D} = \sum_{\substack{p : p \leq n-|S_i| \\ p\geq k+1}}(n-p+1)s_p + |S_i|\sum_{p=n-|S_i|+1}^{k}s_p = L^k_i.\]
 \end{proof}

\section{Proof of Proposition~\ref{prop:lowerbound_outgoing}}
\label{App:lowerbound_outgoing}
\begin{proof}[Proof of Proposition~\ref{prop:lowerbound_outgoing}]
The proof is similar to the proof of Proposition~\ref{Pro:lower}. Since path $P$ has $k$ vertices, we can denote any Hamiltonian cycle $H$ that is extended from $P$ by $H = P \oplus Q$, where $Q$ is a return path that starts at the end vertex of path $P$ and contains $n-k$ remaining vertices in $V_c$. Let $Q = (j_{k+1}, \ldots, j_{n}, j_{n+1}=0)$. Let $T(H)$ denote the total service disruption time of $H$. Let $c_p$ be the number of vertices with no power right before the repair crew arrives at vertex $j_p$ for $p \in [1 : n]$. Let $C=(c_{k+1}, \ldots, c_n)$ and $D = (d_{j_kj_{k+1}}, \ldots, d_{j_{n-1}j_n})$. Then we have
\begin{equation*}
    \begin{split}
        T(H) & = u^F(P) + \sum_{p=k+1}^n c_p d_{j_{p-1}j_p} \\
             & = u^F(P) + C \cdot D \\
             & \ge u^F(P) + f(C) \cdot g(D) \qquad (\text{The Rearrangement Inequality}) \\
             & = u^F(P) + C \cdot g(D) \qquad (\text{The components of $C$ are non-increasing}).
    \end{split}
\end{equation*}

Now we try to find vectors $\bar{C}$ and $\bar{D}$ such that $C \ge \bar{C}$ and $g(D) \ge \bar{D}$ respectively. Note that $c_{k+1} = w_P$, since there are $w_P$ vertices with no power service before vertex $j_{k+1}$ is visited. We have $c_{k+2} \ge n-k-1$, since there at least $n-k-1$ vertices without power services before vertex $j_{k+2}$ is visited, i.e., vertices $j_{k+2}, \ldots, j_n$. Similarly $c_{k+3} \ge n-k-2$, $\ldots$, and $c_n \ge 1$. Thus $C \ge \bar{C} = (w_P, n-k-1, \ldots, 1)$. On the other hand, recall that $s_p$ is the length of the $p^{th}$ shortest arc in $A$. Then $g(D) \ge \bar{D} = (s_1, s_2, \ldots, s_{n-k})$. Therefore, we have
\begin{equation*}
    \begin{split}
        T(H) & \ge u^F(P) + \bar{C} \cdot \bar{D} \\
             & = u^F(P) + w_P s_1 + (n-k-1)s_2 + \ldots + 1 \times s_{n-k} \\
             & = u^F(P) + w_P s_1 + \sum_{p=2}^{n-k} (n-k+1-p)s_p.
    \end{split}
\end{equation*}
\end{proof}

\section{Proof of Proposition~\ref{prop:lowerbound_return}}
\label{App:lowerbound_return}
\begin{proof}[Proof of Proposition~\ref{prop:lowerbound_return}]
The proof is similar to the Proof of Proposition~\ref{prop:lowerbound_outgoing}. Since path $P$ has $k$ vertices, we can denote any Hamiltonian cycle $H$ that is extended from $P$ by $H = Q \oplus P$, where $Q$ is an outgoing path that ends at the starting vertex of path $P$ and contains $n-k$ vertices in $V_c$. Let $Q = (j_0=0, j_1, \ldots, j_{n-k})$. Let $T(H)$ denote the total service disruption time of $H$. Let $c_p$ be the number of vertices with no power right before the repair crew arrives at vertex $j_p$ for $p \in [1 : n]$. Let $C = (c_1, \ldots, c_{n-k+1})$ and $D = (d_{j_0j_1}, \ldots, d_{j_{n-k}j_{n-k+1}})$.
\begin{equation*}
    \begin{split}
        T(H) & = \sum_{p=1}^{n-k+1} c_p d_{j_{p-1}j_p} + v^B(P) \\
             & = C \cdot D + v^B(P) \\
             & \ge f(C) \cdot g(D) + v^B(P) \qquad (\text{The Rearrangement Inequality}) \\
             & = C \cdot g(D) + v^B(P)  \qquad (\text{The components of $C$ are non-increasing}).
    \end{split}
\end{equation*}

Now we try to find vectors $\bar{C}$ and $\bar{D}$ such that $C \ge \bar{C}$ and $g(D) \ge \bar{D}$ respectively. Note that $c_p \ge n-p+1$ since there are at least $n-p+1$ vertices not visited before vertex $j_p$ is visited following $H$, so $C \ge \bar{C} = (n, n-1, \ldots, k)$. On the other hand, recall that $s_p$ is the length of the $p^{th}$ shortest arc in $A$. Then $g(D) \ge \bar{D} = (s_1, s_2, \ldots, s_{n-k})$. Therefore, we have 
\begin{equation*}
    \begin{split}
        T(H) & \ge \bar{C} \cdot \bar{D} + v^B(P) \\
             & = n s_1 + (n-1) s_2 + \ldots + k s_{n-k+1} + v^B(P).
    \end{split}
\end{equation*}
\end{proof}
\end{document}